\documentclass{amsart}            

\begin{document}
\title[Lie groups in the main vertical classes]
{Lie groups as 3-dimensional almost contact B-metric
manifolds in the main vertical classes}

\author{Miroslava Ivanova}

\address[Miroslava Ivanova]{Department of Informatics and Mathematics, Trakia University,
Stara Zagora, 6000, Bulgaria,
E-mail: mivanova@uni-sz.bg}

\newcommand{\ie}{i.e. }
\newcommand{\Id}{\mathrm{Id}}
\newcommand{\X}{\mathfrak{X}}
\newcommand{\W}{\mathcal{W}}
\newcommand{\F}{\mathcal{F}}
\newcommand{\T}{\mathcal{T}}
\newcommand{\LL}{\mathcal{L}}
\newcommand{\TT}{\mathfrak{T}}
\newcommand{\G}{\mathcal{G}}
\newcommand{\I}{\mathcal{I}}
\newcommand{\M}{(M,\allowbreak{}\ff,\allowbreak{}\xi,\allowbreak{}\eta,\allowbreak{}g)}
\newcommand{\Lf}{(G,\ff,\xi,\eta,g)}
\newcommand{\R}{\mathbb{R}}
\newcommand{\N}{\widehat{N}}
\newcommand{\s}{\mathfrak{S}}
\newcommand{\n}{\nabla}
\newcommand{\nn}{\tilde{\nabla}}
\newcommand{\tg}{\tilde{g}}
\newcommand{\ff}{\varphi}
\newcommand{\D}{{\rm d}}
\newcommand{\id}{{\rm id}}
\newcommand{\al}{\alpha}
\newcommand{\bt}{\beta}
\newcommand{\gm}{\gamma}
\newcommand{\dt}{\delta}
\newcommand{\lm}{\lambda}
\newcommand{\ta}{\theta}
\newcommand{\om}{\omega}
\newcommand{\ea}{\varepsilon_\alpha}
\newcommand{\eb}{\varepsilon_\beta}
\newcommand{\eg}{\varepsilon_\gamma}
\newcommand{\sx}{\mathop{\mathfrak{S}}\limits_{x,y,z}}
\newcommand{\norm}[1]{\left\Vert#1\right\Vert ^2}
\newcommand{\nf}{\norm{\n \ff}}
\newcommand{\nN}{\norm{N}}
\newcommand{\Span}{\mathrm{span}}
\newcommand{\grad}{\mathrm{grad}}
\newcommand{\thmref}[1]{The\-o\-rem~\ref{#1}}
\newcommand{\propref}[1]{Pro\-po\-si\-ti\-on~\ref{#1}}
\newcommand{\secref}[1]{\S\ref{#1}}
\newcommand{\lemref}[1]{Lem\-ma~\ref{#1}}
\newcommand{\dfnref}[1]{De\-fi\-ni\-ti\-on~\ref{#1}}
\newcommand{\corref}[1]{Corollary~\ref{#1}}



\numberwithin{equation}{section}
\newtheorem{thm}{Theorem}[section]
\newtheorem{lem}[thm]{Lemma}
\newtheorem{prop}[thm]{Proposition}
\newtheorem{cor}[thm]{Corollary}
\newtheorem{defn}{Definition}[section]

\hyphenation{Her-mi-ti-an ma-ni-fold ah-ler-ian}

\begin{abstract}
Almost contact B-metric manifolds of dimension 3 are
constructed by a two-parametric family of Lie groups.
The class of these manifolds in a known classification of
almost contact B-metric manifolds is
determined as the direct sum of the main vertical classes.
The type of the corresponding Lie algebras in the Bianchi
classification is given. Some geometric characteristics and properties
of the considered manifolds are obtained.
\end{abstract}


\maketitle

\section*{Introduction}

The study of the differential geometry of the almost
contact B-metric manifolds has initiated in \cite{GaMiGr}. The
geometry of these manifolds is a natural extension of the geometry
of the almost complex manifolds with Norden metric \cite{GaBo,
GriMekDjel} in the case of odd dimension. Almost contact
B-metric manifolds are investigated and studied for example in \cite{GaMiGr,
Man4, Man30, Man-Gri1, Man-Gri2, ManIv38, ManIv36, NakGri}.

Here, an object of special interest are the Lie groups considered as 3-dimensional
almost contact B-metric manifolds. For example of such investigation see
\cite{H.Man-Mek}.

The aim of the present paper is to make a study of the most important geometric
characteristics and properties of a family of Lie groups with
almost contact B-metric structure of the lowest dimension 3, belonging to the main vertical classes.
These classes are $\F_4$ and $\F_5$, where the fundamental tensor $F$ is expressed explicitly by the metric $g$, the structure $(\ff,\xi,\eta)$ and the vertical components of the Lee forms $\ta$ and $\ta^*$, i.e. in this case the Lee forms are proportional to $\eta$ at any point. These classes contain some significant examples as the time-like sphere of $g$ and the light cone of the associated metric of $g$ in the complex Riemannian space, considered in \cite{GaMiGr},  as well as the Sasakian-like manifolds studied in \cite{IvManMan45}.

The paper is organized as follows. In Sec.~\ref{mi:sec2}, we give
some necessary facts about almost contact B-metric manifolds. In
Sec.~\ref{mi:sec3}, we construct and study a family of Lie
groups as 3-dimensional manifolds of the considered type.

\section{Almost contact manifolds with B-metric}\label{mi:sec2}

Let $(M,\ff,\xi,\eta,g)$ be a $(2n+1)$-dimensional \emph{almost
contact B-met\-ric manifold}, i.e. $(\ff,\xi,\eta)$ is a triplet
of a tensor (1,1)-field $\ff$, a vector field $\xi$ and
its dual 1-form $\eta$ called an almost
contact structure and the following identities holds:
\begin{equation*}
\ff\xi=0, \quad \ff^2 = -\Id + \eta \otimes \xi, \quad
\eta\circ\ff=0, \quad \eta(\xi)=1,
\end{equation*}
where $\Id$ is the identity. The B-metric $g$ is pseudo-Riemannian and satisfies
\begin{equation*}
g(\ff x,\ff y)=-g(x,y)+\eta(x)\eta(y)
\end{equation*}
for arbitrary tangent vectors $x,y\in T_pM$ at an arbitrary point
$p\in M$ \cite{GaMiGr}.

Further, $x$, $y$, $z$, $w$ will stand for arbitrary vector fields on $M$ or vectors in the tangent space at an arbitrary point in $M$.

Let us note that the restriction of a B-metric on the contact
distribution $H=\ker(\eta)$ coincides with the corresponding Norden
metric with respect to the almost complex structure and the
restriction of $\ff$ on $H$ acts as an anti-isometry on the
metric on $H$ which is the restriction of $g$ on $H$.

The associated metric $\tilde{g}$ of $g$ on $M$ is given by \(
\tilde{g}(x,y)=g(x,\ff y)+\eta(x)\eta(y)\). It is a B-metric, too.
Hence, $(M,\ff,\xi,\eta,\tilde{g})$ is also an almost contact
B-metric manifold. Both metrics $g$ and $\tilde{g}$ are indefinite
of signature $(n+1,n)$.

The structure group of $\M$ is $\G\times\I$, where $\I$ is the
identity on $\Span(\xi)$ and $\G=\mathcal{GL}(n;\mathbb{C})\cap
\mathcal{O}(n,n)$.

The (0,3)-tensor $F$ on $M$ is defined by
$F(x,y,z)=g\bigl( \left( \n_x \ff \right)y,z\bigr)$, where $\n$ is
the Levi-Civita connection of $g$. The tensor $F$ has the
following properties:
\begin{equation*}\label{mi:F-prop}
F(x,y,z)=F(x,z,y)
=F(x,\ff y,\ff z)+\eta(y)F(x,\xi,z)
+\eta(z)F(x,y,\xi).
\end{equation*}

A classification of the almost contact B-metric manifolds is
introduced in \cite{GaMiGr}, where eleven basic classes $\F_i$
$(i=1,\allowbreak{}2,\allowbreak{}\dots,\allowbreak{}{11})$ are
characterized with respect to the properties of $F$. The special
class $\F_0$ is defined by the condition $F(x,y,z)=0$ and is
contained in each of the other classes. Hence, $\F_0$ is the class
of almost contact B-metric manifolds with $\n$-parallel
structures, \ie $\n\ff=\n\xi=\n\eta=\n g=\n \tilde{g}=0$.

Let $g_{ij}$, $i,j\in\{1,2,\dots,2n+1\}$, be the components of
the matrix of $g$ with respect to a basis
$\{e_i\}_{i=1}^{2n+1}=\{e_1,e_2,\dots,e_{2n+1}\}$ of $T_pM$ at an arbitrary point $p\in M$, and $g^{ij}$
-- the components of the inverse  matrix of  $(g_{ij})$. The Lee
forms associated with $F$ are defined as follows:
\begin{equation*}\label{mi:Lee forms}
\ta(z)=g^{ij}F(e_i,e_j,z), \quad \ta^*(z)=g^{ij}F(e_i,\ff e_j,z),
\quad \om(z)=F(\xi,\xi,z).
\end{equation*}

In \cite{Man31}, the \emph{square norm of $\nabla \ff$} is introduced
by:
\begin{equation}\label{mi:snf}
    \norm{\nabla \ff}=g^{ij}g^{ks}
    g\bigl(\left(\nabla_{e_i} \ff\right)e_k,\left(\nabla_{e_j}
    \ff\right)e_s\bigr).
\end{equation}
If $\M$ is an $\F_0$-manifold then the square norm of $\nabla \ff$
is zero, but the inverse implication is not always true. An almost
contact B-metric manifold satisfying the condition $ \norm{\nabla
\ff}=0$ is called an \emph{isotropic-$\F_0$-manifold}. The square
norms of $\n \eta$ and $\n \xi$ are defined in \cite{Man33} by:
\begin{equation}\label{mi:sn eta-xi}
    \norm{\n\eta}=g^{ij}g^{ks}
    \left(\nabla_{e_i} \eta\right)e_k\left(\nabla_{e_j}
    \eta\right)e_s, \quad \norm{\n\xi}=g^{ij}g\left(\nabla_{e_i} \xi,\nabla_{e_j}
    \xi\right).
\end{equation}


Let $R$ be the curvature tensor of type (1,3) of Levi-Civita
connection $\n$, \ie $R(x,y)z=\n_x\n_yz-\n_y\n_xz-\n_{[x,y]}z$.
The corresponding tensor of $R$ of type (0,4) is defined by $R(x,y,z,w)=g(R(x,y)z,w)$.

The Ricci tensor $\rho$ and the scalar curvature $\tau$ for $R$ as
well as their associated quantities are defined by the following traces
$\rho(x,y)=g^{ij}R(e_i,x,y,e_j)$, $\tau=g^{ij}\rho(e_i,e_j)$,
$\rho^{*}(x,y)=g^{ij}R(e_i,x,y,\ff e_j)$ and
$\tau^{*}=g^{ij}\rho^{*}(e_i,e_j)$, respectively.

An almost contact B-metric manifold is called \emph{Einstein} if
the Ricci tensor is proportional to the metric tensor, \ie
$\rho=\lm g,$ $\lm\in\mathbb{R}$.

Let $\al$ be a non-degenerate 2-plane (section) in $T_pM$. It is known from
\cite{NakGri} that the special 2-planes with respect to the almost
contact B-metric structure are: a \emph{totally
real section} if $\al$ is orthogonal to its $\ff$-image $\ff\al$ and $\xi$, a
\emph{$\ff$-holomorphic section} if $\al$ coincides with $\ff\al$
and a \emph{$\xi$-section} if $\xi$ lies on $\al$.

The sectional curvature $k(\al; p)(R)$ of $\al$ with an arbitrary
basis $\{x,y\}$ at $p$ regarding $R$ is defined
by
\begin{equation}\label{mi:sec curv}
k(\al; p)(R)=\frac{R(x,y,y,x)}{g(x,x)g(y,y)-g(x,y)^2}.%
\end{equation}

It is known from \cite{Man31} that a linear connection $D$ is
called a \emph{natural connection} on an arbitrary manifold
$(M,\ff,\allowbreak\xi,\eta,g)$ if the almost contact structure
$(\ff,\xi,\eta)$ and the B-metric $g$ (consequently also
$\tilde{g}$) are parallel with respect to $D$, \ie
$D\ff=D\xi=D\eta=Dg=D\tilde{g}=0$.
%
In \cite{ManIv36}, it is proved that a linear connection $D$ is
natural on $(M,\ff,\allowbreak\xi,\eta,g)$ if and only if
$D\ff=Dg=0$.
%
A natural connection exists on any almost contact B-metric manifold
and coincides with the Levi-Civita connection only the manifold belongs to $\F_0$.

Let $T$ be the torsion tensor of $D$, \ie
$T(x,y)=D_xy-D_yx-[x,y].$ The corresponding tensor of $T$ of type
(0,3) is denoted by the same letter and is defined by the
condition $T(x,y,z)=g(T(x,y),z)$.



In \cite{Man-Gri2}, it is introduced a natural connection
$\dot{D}$ on $\M$ in all basic classes by
\begin{equation}\label{mi:defn-fiB}
\begin{array}{l}
\dot{D}_xy=\n_xy+\frac{1}{2}\bigl\{\left(\n_x\ff\right)\ff
y+\left(\n_x\eta\right)y\cdot\xi\bigr\}-\eta(y)\n_x\xi.
\end{array}
\end{equation}
This connection is called a \emph{$\ff$B-connection} in
\cite{ManIv37}. It is studied for the main classes
$\F_1,\F_4,\F_5,\F_{11}$  in \cite{Man-Gri2, Man3, Man4}. Let us
note that the $\ff$B-connection is the odd-dimensional analogue of
the B-connection on the almost complex manifold with Norden
metric, studied for the class $\W_1$ in \cite{GaGrMi}.

In \cite{ManIv38}, a natural connection $\ddot{D}$ is
called a \emph{$\ff$-canonical connection} on
$(M,\ff,\xi,\allowbreak\eta,g)$ if its torsion tensor $\ddot{T}$
satisfies the following identity:
\begin{equation*}\label{mi:T-can}
\begin{split}
    &\ddot{T}(x,y,z)-\ddot{T}(x,z,y)-\ddot{T}(x,\ff y,\ff z)
    +\ddot{T}(x,\ff z,\ff y)=\\
    &=\eta(x)\left\{\ddot{T}(\xi,y,z)-\ddot{T}(\xi,z,y)
    -\ddot{T}(\xi,\ff y,\ff z)+\ddot{T}(\xi,\ff z,\ff y)\right\}\\
    &+\eta(y)\left\{\ddot{T}(x,\xi,z)-\ddot{T}(x,z,\xi)
    -\eta(x)\ddot{T}(z,\xi,\xi)\right\}\\
    &-\eta(z)\left\{\ddot{T}
(x,\xi,y)-\ddot{T}(x,y,\xi)-\eta(x)\ddot{T}(y,\xi,\xi)\right\}.
\end{split}
\end{equation*}
It is established that the $\ff$B-connection and the
$\ff$-canonical connection coincide if and only if $\M$ is in the class
$\F_1\oplus\F_2\oplus\F_4\oplus\F_5\oplus\F_6\oplus\F_8\oplus\F_9\oplus\F_{10}\oplus\F_{11}$.

In \cite{Man31}, it is introduced and studied one more natural connection on $\M$ called $\ff$KT-connection, which is defined by the condition its torsion to be a 3-form. There is given that the $\ff$KT-connection exists only in the class $\F_3\oplus\F_7$.

In \cite{H.Man} it is determined the class of all 3-dimensional almost contact B-metric manifolds. It is
$
\F_1\oplus\F_4\oplus\F_5\oplus\F_8\oplus\F_9\oplus\F_{10}\oplus\F_{11}.
$

\section{A family of Lie groups as 3-dimensional $(\F_4\oplus\F_5)$-manifolds}\label{mi:sec3}

In this section we study 3-dimensional real connected Lie groups
with almost contact B-metric structure. On a 3-dimensional
connected Lie group $G$ we take a global basis of left-invariant
vector fields $\left\{e_0,e_1,e_2\right\}$ on $G$.

We define an almost contact structure on $G$ by
\begin{equation}\label{mi:f}
\begin{array}{l}
\ff e_0 = o,\quad \ff e_1 = e_2,\quad \ff e_2 =-e_1,\quad \xi = e_0;\\
\eta(e_0)=1,\quad \eta(e_1)=\eta(e_2)=0,
\end{array}
\end{equation}
where $o$ is the zero vector field and define a B-metric on $G$ by
\begin{equation}\label{mi:g}
\begin{array}{l}
g(e_0,e_0)=g(e_1,e_1)=-g(e_2,e_2)=1,
\\
g(e_0,e_1)=g(e_0,e_2)=g(e_1,e_2)=0.
\end{array}
\end{equation}

We consider the Lie algebra $\mathfrak{g}$ on $G$, determined by the
following non-zero commutators:
\begin{equation}\label{mi:komutator}
\left[e_0,e_1\right]=-be_1-ae_2,\quad
\left[e_0,e_2\right]=ae_1-be_2, \quad \left[e_1,e_2\right]=0,
\end{equation}
where $a, b\in\R$.
We verify immediately that the Jacobi identity for $\mathfrak{g}$ is
satisfied. Hence, $G$ is a 2-parametric family of Lie groups with corresponding Lie algebra $\mathfrak{g}$.

\begin{thm}\label{mi:thm F1+F11}
Let $(G,\ff,\xi,\eta,g)$ be a 3-dimensional connected Lie group
with almost contact B-metric structure determined by \eqref{mi:f},
\eqref{mi:g} and \eqref{mi:komutator}. Then it belongs to the
class $\F_4\oplus\F_5$.
\end{thm}

\begin{proof}
The well-known Koszul equality for the Levi-Civita connection $\n$ of $g$
\begin{equation}\label{mi:Koszul}
2g\left(\n_{e_i}e_j,e_k\right)=g\left(\left[e_i,e_j\right],e_k\right)
+g\left(\left[e_k,e_i\right],e_j\right)+g\left(\left[e_k,e_j\right],e_i\right)
\end{equation}
implies the following form of the components
$F_{ijk}=F(e_i,e_j,e_k)$ of $F$:
\begin{equation}\label{mi:F}
\begin{split}
 2F_{ijk}=g\left(\left[e_i,\ff e_j\right]-\ff
\left[e_i,e_j\right],e_k\right)&+g\left(\ff
\left[e_k,e_i\right]-\left[\ff e_k,e_i\right],e_j\right)\\[4 pt]
&+g\left(\left[e_k,\ff e_j\right]-\left[\ff
e_k,e_j\right],e_i\right).
\end{split}
\end{equation}
Using \eqref{mi:F} and \eqref{mi:komutator} for the non-zero
components $F_{ijk}$, we get:
\begin{equation}\label{mi:Fijk}
\begin{array}{l}
F_{101}=F_{110}=-F_{202}=-F_{220}=a, \\[4 pt]
F_{102}=F_{120}=F_{201}=F_{210}=b.
\end{array}
\end{equation}
Immediately we establish that the components in
\eqref{mi:Fijk} satisfy the condition $F=F^4+F^5$ which means that the manifold belongs to $\F_4\oplus\F_5$. Here, the components $F^s$
of $F$ in the basic classes $\F_s$ $(s=4,5)$ have the following form (see \cite{H.Man})
\begin{equation}\label{Fi3}
\begin{array}{l}
F_{4}(x,y,z)=\frac{1}{2}\ta_0\Bigl\{x^1\left(y^0z^1+y^1z^0\right)
-x^2\left(y^0z^2+y^2z^0\right)\bigr\},\\[4pt]
%
\phantom{F_{4}(x,y,z)=}
\frac{1}{2}\ta_0=F_{101}=F_{110}=-F_{202}=-F_{220};\\[4pt]
F_{5}(x,y,z)=\frac{1}{2}\ta^*_0\bigl\{x^1\left(y^0z^2+y^2z^0\right)
+x^2\left(y^0z^1+y^1z^0\right)\bigr\},\\[4pt]
%
\phantom{F_{5}(x,y,z)=}
\frac{1}{2}\ta^*_0=F_{102}=F_{120}=F_{201}=F_{210}.
\end{array}
\end{equation}
where $\ta_0=\ta(e_0)$ and $\ta^*_0=\ta^*(e_0)$ are determined by $\ta_0=2a$, $\ta^*_0=2b$.
Therefore, the induce 3-dimensional manifold
$(G,\ff,\xi,\eta,g)$ belongs to the class $\F_4\oplus\F_5$ from
the mentioned classification. It is an $\F_0$-manifold if and only
if $(a,b)=(0,0)$ holds.

Obviously, $(G,\ff,\xi,\eta,g)$ belongs to $\F_4$, $\F_{5}$ and $\F_0$ if and only
if the parameters $\ta_0$, $\ta^*_0$ and $\ta_0=\ta^*_0$ vanish, respectively.

According to the above, the commutators in
\eqref{mi:komutator} take the form
\begin{equation}\label{mi:komutator=}
\begin{array}{l}
\left[e_0,e_1\right]=-\frac12(\ta^*_0e_1+\ta_0e_2),\quad
\left[e_0,e_2\right]=\frac12(\ta_0e_1-\ta^*_0e_2), \quad \left[e_1,e_2\right]=0,
\end{array}
\end{equation}
in terms of the basic components of the Lee forms $\ta$ and $\ta^*$.
\end{proof}

According to \thmref{mi:thm F1+F11} and the consideration in \cite{H.Man_Bia}, we can remark that the Lie algebra determined as above belongs to the type $Bia(VII_h)$, $h>0$ of the
Bianchi classification (e.g. \cite{Bia1, Bia2}).

Using \eqref{mi:Koszul} and \eqref{mi:komutator}, we obtain the
components of $\n$:
\begin{equation}\label{mi:nabli}
\begin{array}{l}
\n_{e_1}e_0=be_1+ae_2, \quad \n_{e_1}e_1=-be_0, \quad
\n_{e_1}e_2=ae_0,\\[4 pt]%
\n_{e_2}e_0=-ae_1+be_2, \quad \n_{e_2}e_1=ae_0, \quad
\n_{e_2}e_2=be_0.
\end{array}
\end{equation}

We denote by $R_{ijkl}=R(e_i,e_j,e_k,e_l)$ the components of the
curvature tensor $R$, $\rho_{jk}=\rho(e_j,e_k)$ of the Ricci
tensor $\rho$, $\rho^*_{jk}=\rho^*(e_j,e_k)$ of the associated
Ricci tensor $\rho^*$ and $k_{ij}=k(e_i,e_j)$ of the sectional curvature for
$\n$ of the basic 2-plane $\al_{ij}$ with a basis $\{e_i,e_j\}$,
where $i, j\in\{0, 1, 2\}$. On the considered manifold
$(G,\ff,\xi,\eta,g)$ the basic 2-planes $\al_{ij}$ of special type are:
a $\ff$-holomorphic section
--- $\al_{12}$ and $\xi$-sections --- $\al_{01}$, $\al_{02}$. Further, by \eqref{mi:sec curv}, \eqref{mi:g}, \eqref{mi:komutator} and
\eqref{mi:nabli}, we compute
\begin{equation}\label{mi:R-F1+F11}
\begin{array}{c}
-R_{0101}=R_{0202}=\frac{1}{2}\rho_{00}=k_{01}=k_{02}=\frac14(\ta_0^2-\ta^{*2}_0),\\[4 pt]
R_{0102}=R_{0201}=-\rho_{12}=-\frac{1}{2}\rho^*_{00}=-\frac{1}{2}\tau^*=-\frac12\ta_0\ta^{*}_0,\\[4 pt]
R_{1212}=\rho^*_{12}=k_{12}=-\frac14(\ta_0^2+\ta^{*2}_0), \quad \rho_{11}=-\rho_{22}=-\frac12\ta^{*2}_0,\\[4 pt]
\tau=\frac12(\ta_0^2-3\ta^{*2}_0).
\end{array}
\end{equation}
The rest of non-zero components of $R$, $\rho$ and $\rho^*$ are
determined by \eqref{mi:R-F1+F11} and the properties
$R_{ijkl}=R_{klij},$ $R_{ijkl}=-R_{jikl}=-R_{ijlk},$
$\rho_{jk}=\rho_{kj}$ and $\rho^*_{jk}=\rho^*_{kj}.$

Taking into account \eqref{mi:snf}, \eqref{mi:sn eta-xi},
\eqref{mi:f}, \eqref{mi:g} and \eqref{mi:nabli}, we have
\begin{equation}\label{mi:norm fi}
\norm{\nabla \ff}=-2\norm{\nabla \eta}=-2\norm{\nabla
\xi}=\ta_0^2-\ta^{*2}_0.
\end{equation}

\begin{prop}
The following characteristics are valid for $(G,\ff,\xi,\eta,g)$:
\begin{enumerate}
\item \label{1}
The $\ff$B-connection $\dot D$ (respectively,
$\ff$-canonical connection $\ddot D$) is zero in the basis $\{e_0,e_1,e_2\}$;
\item \label{3}
The manifold is an isotropic-$\F_0$-manifold if and only if
the condition $\ta_0=\pm\ta^{*}_0$ is valid.
\item \label{5}
The manifold is flat if and only if it belongs to $\F_0$;
\item \label{6}
The manifold is Ricci-flat
(respectively, $*$-Ricci-flat) if and only if
it is flat.
\item \label{4}
The manifold is scalar flat if and only if the condition
$\ta_0=\pm\sqrt{3}\,\ta^{*}_0$ holds.
\item \label{7}
The manifold is $*$-scalar flat if
and only if it belongs to either $\F_4$ or $\F_5$.
\end{enumerate}
\end{prop}

\begin{proof}
Using \eqref{mi:defn-fiB}, \eqref{mi:f} and \eqref{mi:nabli}, we
get immediately the assertion (\ref{1}).
Equation \eqref{mi:norm fi} imply the assertion 
(\ref{3}).
The assertions (\ref{4}), (\ref{5}) and (\ref{7}) holds, according to
\eqref{mi:R-F1+F11}.
On the 3-dimensional almost contact B-metric manifold with the
basis $\left\{e_0,e_1,e_2\right\}$, bearing in mind the definitions
of the Ricci tensor $\rho$ and the $\rho^*$,  we have
\[
\rho_{jk}=R_{0jk0}+R_{1jk1}-R_{2jk2} \qquad
\rho^*_{jk}=R_{1kj2}+R_{2jk1}.
\]
By virtue of the latter equalities, we get the assertion (\ref{6}).
\end{proof}

According to \eqref{mi:Fijk}, \eqref{mi:R-F1+F11} and
\eqref{mi:norm fi}, we establish the truthfulness of the following
\begin{prop}\label{mm:prop}
The following properties are equivalent for the studied manifold $(G,\ff,\xi,\eta,g)$:
\begin{enumerate}
\item it belongs to $\F_5$;
\item it is Einstein;
\item it is a hyperbolic space form with $k=-\frac14\ta^{*2}_0$;
\item it is $*$-scalar flat;
\item the Lee form  $\ta$ vanishes.
\end{enumerate}
\end{prop}



\end{document}